\documentclass[12pt]{amsart}
\usepackage{amssymb}
\usepackage{amsmath}
\usepackage[active]{srcltx}
\usepackage{t1enc}
\usepackage[utf8]{inputenc}
\usepackage{verbatim}
\usepackage{amsmath,amsfonts,amssymb,amsthm}
\usepackage[mathcal]{eucal}
\usepackage{enumerate}
\usepackage[centertags]{amsmath}
\usepackage{graphics}

\newtheorem{theorem}{Theorem}

\newtheorem{lemma}{Lemma}
\newtheorem{remark}{Remark}

\newtheorem{corollary}{Corollary}

\begin{document}
\author{David Baramidze$^{1,2}$,  Lars-Erik Persson$^{2,3 ^{\star}}$, Kristoffer Tangrand$^{2}$, George Tephnadze$^{1}$}

\title[Nörlund means \dots]{$(H_p-L_p)$ type inequalities for subsequences of Nörlund means of Walsh-Fourier series}
\date{}
\maketitle

$^{1}$The University of Georgia, School of Science and Technology, 77a Merab Kostava St, Tbilisi 0128, Georgia.

$^{2}$Department of Computer Science and Computational Engineering,  UiT The Arctic University of Norway, P.O. Box 385, N-8505, Narvik, Norway.

$^{3}$Department of Mathematics and Computer Science, Karlstad University, 65188 Karlstad, Sweden.

$^{\star}$Corresponding author: Kristoffer Tangrand  (email: ktangrand@gmail.com)
\begin{abstract}
We investigate the subsequence $\{t_{2^n}f \}$ of N\"{o}rlund means with respect to the Walsh
system generated by non-increasing and convex sequences. In particular, we prove that a big class of such summability methods are not bounded  from the martingale Hardy spaces $H_p$ to the  space $weak-L_p $ for $0<p<1/(1+\alpha) $, where $0<\alpha<1$. Moreover, some new related inequalities are derived. As application, some well-known and new results are pointed out for well-known summability methods, especially for N\"{o}rlund logarithmic means and Ces\`aro means.
\end{abstract}

\date{}



\section{Introduction}

The terminology and notations used in this introduction can be found in Section 2. 

The fact that the Walsh system is the group of characters of a compact abelian group connects Walsh analysis with abstract harmonic analysis was discovered independently by Fine \cite{fi} and Vilenkin \cite{Vi}. For general references to the Haar measure and harmonic analysis  on groups   see Pontryagin \cite{Pontryagin}, Rudin \cite{Rudin}, and Hewitt and Ross \cite{HR}. In particular, Fine investigated the group $G$, which is a direct product of the additive groups $Z_{2}=:\{0,1\}$  and introduced the Walsh system $\{{w}_j\}_{j=0}^{\infty}$. 

It is well-known that Walsh systems do not form bases in the
space $L_{1}.$ Moreover, there is a function in the Hardy space $H_{1},$
such that the partial sums of $f$ are not bounded in the $L_{1}$-norm. Moreover, (see \cite{tep7}) there
exists a martingale $f\in H_{p}\left( 0<p<1\right),$ such that
$$
\underset{n\in \mathbb{N}}{\sup }\left\Vert S_{2^n+1}f\right\Vert
_{weak-L_{p}}=\infty.
$$
On the other hand, (for details see e.g. the books \ \cite{sws}
and \cite{We1} and especially the newest one \cite{PTWbook}) the subsequence $\{S_{2^n}\}$ of partial
sums is bounded from the martingale Hardy space $H_{p}$ to the
space $H_{p},$ for all $p>0,$ that is the following inequality holds:
\begin{equation}\label{snjemala}
\left\Vert S_{2^n}f\right\Vert _{H_p}\leq c_{p}\left\Vert f\right\Vert _{H_p}, \ \ n\in \mathbb{N}, \ \ p>0.
\end{equation}

Weisz \cite{We3} proved that Fej\'er means of
Vilenkin-Fourier series are bounded from the martingale Hardy space $H_{p}$ to the
space $H_{p},$ for  $p>1/2.$ Goginava \cite{gog1} (see also \cite{PTT}, \cite{NT1,NT2,NT3,NT4}) proved that there exists a
martingale $f\in H_{1/2}$ such that
$$
\sup_{n\in \mathbb{N}}\left\Vert \sigma _{n}f\right\Vert _{1/2}=+\infty .
$$
However, Weisz \cite{We3} (see also \cite{pt}) proved that for every $ f\in H_p, $ there exists an absolute constant $ c_p, $ such that the following inequality holds:
\begin{equation} \label{sigmanjemala}
\left\Vert \sigma_{2^n}f\right\Vert _{H_p}\leq c_{p}\left\Vert f\right\Vert _{H_p}, \ \ n\in \mathbb{N}, \ \ p>0.
\end{equation}

M\'oricz and Siddiqi \cite{Mor} investigated the approximation properties of some special N\"orlund means of Walsh-Fourier series of $L_{p}$ functions in
norm. Approximation properties for general summability methods can be found in \cite{BN,BNT}. Fridli, Manchanda and Siddiqi \cite{FMS} improved and extended the results of M\'oricz and Siddiqi \cite{Mor} to martingale Hardy spaces.
The case when $\left\{ q_{k}=1/k:k\in \mathbb{N}\right\} $ was
excluded, since the methods are not applicable to N\"o%
rlund logarithmic means. In \cite{Ga2} G\'{a}t and Goginava proved some
convergence and divergence properties of the N\"orlund logarithmic means of
functions in the Lebesgue space $%
L_1.$ In particular, they proved that there exists a function $f$ in the space $ L_1, $ such that 
$$
\sup_{n\in \mathbb{N}}\left\Vert L_{n}f\right\Vert _{1}=\infty .
$$
In \cite{BPT} (see also \cite{PTW}) it was proved that there exists a martingale $f\in H_{p}, \ \ (0<p< 1)$ such that
$$
\sup_{n\in \mathbb{N}}\left\| L_{2^n}f\right\| _{p}=\infty .
$$
A counterexample for $p=1$ was proved in \cite{PTW2}. 
However, Goginava \cite{gog2} proved that  for every $ f\in H_1, $ there exists an absolute constant $ c, $ such that the following inequality holds:
\begin{equation} \label{gogjemala}
\left\Vert L_{2^n}f\right\Vert _{1}\leq c\left\Vert f\right\Vert _{H_1}, \ \ n\in \mathbb{N}.
\end{equation}

In  \cite{BPST} it was proved that for any $0<p<1,$ there exists a martingale $f\in H_{p}$ such that 
$$
\sup_{n\in \mathbb{N}} \left\Vert L_{2^n}f\right\Vert_{weak-L_p}=\infty.
$$

In \cite{PTW1} is was proved that for any  non-decreasing sequence $(q_k,k\in \mathbb{N})$ satisfying the conditions 
\begin{equation} \label{6a}
\frac{1}{Q_n}=O\left(\frac{1}{n^{\alpha}}\right) \text{ \ \ \ and \ \ \ }
q_n-q_{n+1}=O\left(\frac{1}{n^{2-\alpha}}\right) ,\text{ \  as \  }
n\rightarrow \infty,
\end{equation}
then, for every $ f\in H_p, $ where $p>1/(1+\alpha),$ there exists an absolute constant $ c_p, $ depending only on $p,$ such that the following inequality holds:
\begin{equation} \label{weaktypenorlund}
\left\Vert t_{n}f\right\Vert _{H_p}\leq c_{p}\left\Vert f\right\Vert _{H_p}, \ \ n\in \mathbb{N}.
\end{equation}
Boundedness does not hold from $H_{p}$ to $weak-L_{p},$ for $0<p< 1/ (1+\alpha).$ 
As a consequence, (for details see \cite{we6}) we get that the Ces\`aro means $\sigma_n^{\alpha}$ is bounded from $H_{p}$ to $L_{p},$ for  $p>1/(1+\alpha),$ but they are not bounded from $H_{p}$ to $weak-L_{p},$ for $0<p< 1/ (1+\alpha).$ In the endpoint case $p=1/ (1+\alpha),$ Weisz and Simon \cite{sw} proved that the maximal operator $\sigma ^{\alpha ,\ast }$  of Ces\`aro means define by
$$\sigma ^{\alpha ,\ast }f:=\sup_{n\in\mathbb{N}}\vert\sigma ^{\alpha}_nf\vert$$
 is bounded
from the Hardy space $H_{1/\left(1+\alpha \right) }$ to the space $weak-L_{1/\left(1+\alpha \right)}.$ Goginava \cite{gog4} gave a counterexample,
which shows that boundedness does not hold for $0<p\leq 1/\left(1+\alpha
\right) .$

In this paper we  develop some methods considered in \cite{BPTW,BPST,LPTT} (see also the new book \cite{PTWbook}) and prove that for any $0<p<1,$ there exists a martingale $f\in H_{p}$ such that 
$$
\sup_{n\in \mathbb{N}} \left\Vert t_{2^n}f\right\Vert_{weak-L_p}=\infty.
$$
Moreover, we prove that a big class of subsequence $\{t_{2^n}f \}$ of Nörlund means with respect to the Walsh
system generated by non-increasing and convex sequences  are not bounded  from the martingale Hardy spaces $H_p$ to the  space $weak-L_p $ for $0<p<1/(1+\alpha) $, where $0<\alpha<1$. Moreover, some new related inequalities are derived. As application, some well-known and new results are pointed out for well-known summability methods, especially for Nörlund logarithmic means and Ces\`aro means.

The main results in this paper are presented and proved in Section 4. Section 3 is used to present some auxiliary results, where, in particular, Lemma 2 is new and of independent interest. In order not to disturb our discussions later on some definitions and notations are given in Section 2. 

\section{Definitions and Notations}

\bigskip Let $\mathbb{N}_{+}$ denote the set of the positive integers, $%
\mathbb{N}:=\mathbb{N}_{+}\cup \{0\}.$ Denote by $Z_{2}$ the discrete cyclic
group of order 2, that is $Z_{2}:=\{0,1\},$ where the group operation is the
modulo 2 addition and every subset is open. The Haar measure on $Z_2$ is
given so that the measure of a singleton is 1/2.

Define the group $G$ as the complete direct product of the group $Z_{2},$
with the product of the discrete topologies of $Z_{2}$`s. 

The elements of $G$
are represented by sequences 
$$x:=(x_{0},x_{1},...,x_{j},...), \ \ \ \text{ where } \ \ \ 
x_{k}=0\vee 1.$$

It is easy to give a base for the neighborhood of $x\in G$ namely:
\begin{equation*}
I_{0}\left( x\right) :=G,\text{ \ }I_{n}(x):=\{y\in
G:y_{0}=x_{0},...,y_{n-1}=x_{n-1}\}\text{ }(n\in \mathbb{N}).
\end{equation*}

Denote $I_{n}:=I_{n}\left( 0\right) ,$ $\overline{I_{n}}:=G$ $\backslash $ $%
I_{n}$ and 
$$e_{n}:=\left( 0,...,0,x_{n}=1,0,...\right) \in G, \ \ \text{ for } n\in 
\mathbb{N}.$$

If $n\in \mathbb{N},$ then every $n$ can be uniquely expressed as 
$n=\sum_{k=0}^{\infty }n_{j}2^{j},$  where $ n_{j}\in Z_{2} \   (j\in \mathbb{N}) $
and only a finite numbers of $n_{j}$ differ from zero. Let 
$$\left\vert n\right\vert :=\max \{k\in \mathbb{N}:\ n_{k}\neq 0\}.$$

The norms (or quasi-norms) of the spaces $L_{p}(G)$ and $weak-L_{p}\left(
G\right) ,$ $\left( 0<p<\infty \right) $ are, respectively, defined by 
\begin{equation*}
\left\Vert f\right\Vert _{p}^{p}:=\int_{G}\left\vert f\right\vert ^{p}d\mu \ \ \ \text{
and}
\ \ \ 
\left\Vert f\right\Vert_{weak-L_{p}}^{p}:=\sup_{\lambda
	>0}\lambda ^{p}\mu \left( f>\lambda \right) .
\end{equation*}

The $k$-th Rademacher function is defined by
\begin{equation*}
r_{k}\left( x\right) :=\left( -1\right) ^{x_{k}}\text{\qquad }\left( \text{ }%
x\in G,\text{ }k\in \mathbb{N}\right) .
\end{equation*}

Now, define the Walsh system $w:=(w_{n}:n\in \mathbb{N})$ on $G$ as: 
\begin{equation*}
w_{n}(x):=\overset{\infty }{\underset{k=0}{\Pi }}r_{k}^{n_{k}}\left(
x\right) =r_{\left\vert n\right\vert }\left( x\right) \left( -1\right) ^{
\underset{k=0}{\overset{\left\vert n\right\vert -1}{\sum }}n_{k}x_{k}}\text{
\qquad }\left( n\in \mathbb{N}\right) .
\end{equation*}

It is well-known that  (see e.g. \cite{sws}) the Walsh system is orthonormal and complete in $L_{2}\left( G\right) .$ Moreover, for any $n\in\mathbb{N},$
\begin{eqnarray}\label{vilenkin}
w_n\left( x+y\right) &=&w_n\left( x\right)w_n\left( y\right).
\end{eqnarray}

If $f\in L_{1}\left( G\right) $ we define the Fourier coefficients, partial sums and  Dirichlet kernel by
\begin{eqnarray*}
\widehat{f}\left( k\right) &:=&\int_{G}fw_{k}d\mu \,\,\,\,\left( k\in \mathbb{N
}\right) ,\\
S_{n}f&:=&\sum_{k=0}^{n-1}\widehat{f}\left( k\right) w_{k},\ \ \ 
D_{n}:=\sum_{k=0}^{n-1}w_{k\text{ }}\,\,\,\left( n\in \mathbb{N}_{+}\right).
\end{eqnarray*}

Recall that (for details see e.g. \cite{sws}):
\begin{equation}
D_{2^{n}}\left( x\right) =\left\{ 
\begin{array}{ll}
2^{n}, & \,\text{if\thinspace \thinspace \thinspace }x\in I_{n} \\ 
0, & \text{if}\,\,x\notin I_{n}%
\end{array}%
\right.  \label{1dn}
\end{equation}%
and
\begin{equation}
D_{n}=w_{n}\overset{\infty }{\underset{k=0}{\sum }}n_{k}r_{k}D_{2^{k}}=w_{n}%
\overset{\infty }{\underset{k=0}{\sum }}n_{k}\left(
D_{2^{k+1}}-D_{2^{k}}\right),\text{ for  }n=\overset{\infty }{\underset{i=0%
}{\sum }}n_{i}2^{i}.  \label{2dn}
\end{equation}

Let $\left\{ q_{k}, \ k\geq 0\right\} $ be a sequence of nonnegative numbers.
The Nörlund means for the Fourier series of $f$ are defined by 
\begin{equation*}
t_nf:=\frac{1}{Q_n}\sum_{k=1}^{n}q_{n-k}S_{k}f, \ \ \ \text{where} \ \ \ 
Q_{n}:=\sum_{k=0}^{n-1}q_k. 
\end{equation*}

In this paper we consider convex  $\left\{ q_{k}, \ k\geq 0\right\} $ sequences, that is 
$$q_{n-1}+q_{n+1}-2q_n\geq 0, \ \ \ \text{for all} \ \ \ n\in \mathbb{N }.$$

If the function $\psi(x)$ is any real valued and  convex function (for example $\psi(x)=x^{\alpha-1}, \ 0\leq \alpha\leq 1$), then the sequence  $\{\psi(n), \ n\in\mathbb{N }\}$ is convex. 

Since 
$q_{n-2}-q_{n-1}\geq q_{n-1}-q_{n}\geq q_{n}-q_{n+1}\geq q_{n+1}-q_{n+2}$
we find that $$q_{n-2}+q_{n+2}\geq q_{n-1}+q_{n+1}$$
and we also get that
\begin{equation}\label{11}
q_{n-2}+q_{n+2}-2q_n\geq 0, \ \ \ \text{for all} \ \ \ n\in \mathbb{N }.
\end{equation}

In the special case when $\{q_{k}=1, \ k\in \mathbb{N}\},$ we have the Fej\'er means
\begin{equation*}
\sigma _{n}f:=\frac{1}{n}\sum_{k=1}^{n}S_{k}f.
\end{equation*}
Moreover, if $q_{k}={1}/{(k+1)}$, then we get the Nörlund logarithmic means: 
\begin{equation}\label{L_n}
L_{n}f:=\frac{1}{l_{n}}\sum_{k=1}^{n}\frac{S_{k}f}{n+1-k}, \ \ \   
\ \ \ \text{where} \ \ \ l_{n}:=\sum_{k=1}^{n}\frac{1}{k}. 
\end{equation}

The Ces\`aro means $\sigma_n^{\alpha}$  (sometimes also denoted $\left(C,\alpha\right)$) is also well-known example of Nörlund means defined by
\begin{equation*}
\sigma_n^{\alpha}f=:\frac{1}{A_n^{\alpha}}\overset{n}{\underset{k=1}{\sum}}A_{n-k}^{\alpha-1}S_kf,
\end{equation*} 
where 
\begin{equation*}
A_0^{\alpha}:=0,\qquad A_n^{\alpha}:=\frac{\left(\alpha+1\right)\ldots\left(\alpha+n\right)}{n!},\qquad \alpha \neq -1,-2,\ldots
\end{equation*}
It is well-known that  
\begin{equation} \label{node0}
A_n^{\alpha}=\overset{n}{\underset{k=0}{\sum}}A_{n-k}^{\alpha-1}, \ \ \  A_n^{\alpha}-A_{n-1}^{\alpha}=A_n^{\alpha-1}\ \ \ \text{and} \ \ \
A_{n}^{\alpha }\sim n^{\alpha }.
\end{equation}

We also define $U_n^{\alpha}$ means as
\begin{equation*}
U^\alpha_nf:=\frac{1}{Q_n}\overset{n}{\underset{k=1}{\sum}}{(n+1-k)}^{(\alpha-1)} S_kf
\ \ \ \text{where} \ \ \ Q_{n}:=\sum_{k=1}^{n}k^{\alpha-1}. 
\end{equation*}

Let us also define $V_n^{\alpha}$ means as
\begin{equation*}
V_nf:=\frac{1}{Q_n}\overset{n}{\underset{k=1}{\sum}}{\ln(n+1-k)}S_kf
\ \ \ \text{where} \ \ \ Q_{n}:=\sum_{k=1}^{n}\frac{1}{\ln (k+1)}. 
\end{equation*}

%Riesz logarithmic means  are defined by 
%\begin{equation*}
%R_{n}f:=\frac{1}{l_{n}}\sum_{k=1}^{n}\frac{S_{k}f}{k}, \ \ l_{n}:=\sum_{k=1}^{n}%
%\frac{1}{k},
%\end{equation*}%
%We note that it is an  inverse of the Nörlund logarithmic means.

Let $f:=\left( f^{\left( n\right) },n\in \mathbb{N}%
\right) $ be a martingale with respect to $\digamma _{n}\left( n\in \mathbb{N%
}\right), $  which are generated by the intervals $\left\{ I_{n}\left(
x\right) :x\in G\right\}$ (for details see e.g. \cite{We1}).

We say that this martingale belongs to the Hardy martingale spaces $H_{p}\left( G\right), $ where $0<p<\infty, $ if
$$
\left\| f\right\| _{H_{p}}:=\left\| f^{*}\right\|_{p}<\infty, \ \ \ 
\text{with} \ \ \
f^{\ast }:=\sup_{n\in \mathbb{N}}\left\vert f^{(n)}\right\vert .
$$

In the case $f\in L_{1}\left(G\right),$ the maximal functions are also 
given by 
\begin{equation*}
M(f)\left( x\right) :=\sup\limits_{n\in \mathbb{N}}\left( \frac{1}{\mu
	\left( I_{n}\left( x\right) \right) }\left\vert \int_{I_{n}\left( x\right)
}f\left( u\right) d\mu \left( u\right) \right\vert \right) .
\end{equation*}

If $f\in L_{1}\left( G\right),$ then it is easy to show that the sequence $F=
\left( S_{2^{n}}f :n\in \mathbb{N}\right) $ is a martingale and $F^*=M(f).$

If $f=\left( f^{\left( n\right) },\text{ }n\in \mathbb{N}\right) $ is a
martingale, then the Walsh-Fourier coefficients must be defined in a slightly
different manner:
\begin{equation*}
\widehat{f}\left( i\right) :=\lim_{k\rightarrow \infty }\int_{G}f^{\left(
k\right) }\left( x\right)w_{i}\left( x\right) d\mu \left( x\right) .
\end{equation*}

A bounded measurable function $a$ is $p$-atom, if there exists an interval $I,$
such that
\begin{equation*}
\text{ supp}\left( a\right) \subset I, \ \ \ \int_{I}ad\mu =0\text{ \ and \ }\left\Vert a\right\Vert _{\infty }\leq \mu \left(
I\right) ^{-1/p}.
\end{equation*}

\section{Auxiliary Results}

The Hardy martingale space $H_{p}\left( G\right) $ has an atomic
characterization (see Weisz \cite{We1}, \cite{We3}):
\begin{lemma} \label{Weisz}
A martingale $f=\left( f^{\left( n\right) },\ n\in \mathbb{N}
\right) $ is in $H_{p}\left( 0<p\leq 1\right) $ if and only if there exist a
sequence $\left( a_{k},k\in \mathbb{N}\right) $ of p-atoms
and a sequence $
\left( \mu _{k},k\in\mathbb{N}\right) $ of real numbers such that for
every $n\in \mathbb{N}:$
\begin{equation}\label{1}
\underset{k=0}{\overset{\infty }{\sum }}\mu _{k}S_{2^{n}}a_{k}=f^{\left(
n\right) },   \ \ \ \ \text{where} \ \ \ \ 
\sum_{k=0}^{\infty }\left| \mu _{k}\right| ^{p}<\infty . 
\end{equation}
Moreover, the following two-sided inequality holds
$$
\left\| f\right\| _{H_{p}}\backsim \inf \left( \sum_{k=0}^{\infty }\left|
\mu _{k}\right| ^{p}\right) ^{1/p},
$$
where the infimum is taken over all decompositions of $f$ of the form (\ref{1}).
\end{lemma}

We also state and prove the following new lemma of independent interest:
\begin{lemma}
Let $k\in \mathbb{N},$  $\left\{ q_{k}:k\in \mathbb{N}\right\}$ be any convex and non-increasing sequence and $x\in I_{2}(e_{0}+e_{1})\in I_{0}\backslash I_{1}.$ Then, for any $\{\alpha_k\},$ the following inequality holds:
\begin{eqnarray*}
\left\vert\sum_{j=2^{2\alpha_k}}^{2^{2\alpha _k+1}}q_{2^{2\alpha _k+1}-j}{D_j}
\right\vert  \geq q_1-\frac{3}{2}q_3.
\end{eqnarray*}
\end{lemma}
\begin{proof}
Let $x\in I_{2}(e_{0}+e_{1})\in I_{0}\backslash I_{1}.$ According to (\ref%
{1dn}) and (\ref{2dn}) we get that
\begin{equation*}
D_{j}\left( x\right) =\left\{ 
\begin{array}{ll}
w_{j}, & \,\text{if\thinspace \thinspace \thinspace }j\ \ \text{is odd
	number,} \\ 
0, & \text{if}\,\,j\ \ \text{is even number,}%
\end{array}%
\right.
\end{equation*}
and 
\begin{eqnarray*}
\sum_{j=2^{2\alpha_k}}^{2^{2\alpha _k+1}-1}q_{2^{2\alpha _k+1}-j}D_j=\sum_{j=2^{2\alpha _{k}-1}}^{2^{2\alpha _{k}}-1}q_{2^{2\alpha _{k}+1}-2j-1}{w_{2j+1}}=w_1\sum_{j=2^{2\alpha _{k}-1}}^{2^{2\alpha _{k}}-1}q_{
		2^{2\alpha _{k}+1}-2j-1}{w_{2j}}.
\end{eqnarray*}
By using \eqref{11} we find that
\begin{eqnarray*}
&&\sum_{j=2^{2\alpha _{k}-2}+1}^{2^{2\alpha _{k}-1}-1}\left\vert q_{2^{2\alpha _{k}+1}-4j+3}-q_{2^{2\alpha _{k}+1}-4j+1}\right\vert \\
&=&\sum_{j=2^{2\alpha _{k}-2}+1}^{2^{2\alpha _{k}-1}-1}\left( q_{
2^{2\alpha _{k}+1}-4j+1}-q_{2^{2\alpha _{k}+1}-4j+3}\right) \\
&=&\left(q_{2^{2\alpha _{k}}-3}-q_{2^{2\alpha _{k}}-1}\right)+\left( q_{
2^{2\alpha _{k}}-7}-q_{2^{2\alpha _{k}}-5}\right)+\ldots+\left( q_{5}-q_{7}\right)\\
&\leq&\frac12\left( q_{2^{2\alpha _{k}}-3}-q_{2^{2\alpha _{k}}-1}\right)+
\frac12\left( q_{2^{2\alpha _{k}}-5}-q_{2^{2\alpha _{k}}-3}\right)\\
&+&\frac12\left( q_{2^{2\alpha _{k}}-7}-q_{2^{2\alpha _{k}}-5}\right)
+\frac12\left( q_{2^{2\alpha _{k}}-9}-q_{2^{2\alpha _{k}}-7}\right)\\
&+&\ldots
+\frac12\left( q_{5}-q_{7}\right)+\frac12\left( q_{3}-q_{5}\right)\leq \frac12q_{3}-\frac12q_{2^{2\alpha _{k}}-1}.
\end{eqnarray*}
Hence, if we apply 
$$
w_{4k+2}=w_{2}w_{4k}=-w_{4k}, \ \ \text{for} \ \ x\in I_{2}(e_{0}+e_{1}),
$$
we find that

\begin{eqnarray*}
&&\left\vert\sum_{j=2^{2\alpha_k}}^{2^{2\alpha _k+1}-1}q_{2^{2\alpha _k+1}-j}{D_j}\right\vert\\
&=&\left\vert q_0w_{2^{2\alpha_{k}+1}-2}+q_3{w_{2^{2\alpha_{k}+1}-4}}+\sum_{j=2^{2\alpha _{k}-1}}^{2^{2\alpha _{k}}-1}q_{2^{2\alpha _{k}+1}-2j-1}{w_{2j}}\right\vert  \\
&=&\left\vert (q_3-q_1){2w_{2^{2\alpha _{k}+1}-4}} +\sum_{j=2^{2\alpha_{k}-2}+1}^{2^{2\alpha _{k}-1}}\left( q_{2^{2\alpha _{k}+1}-4j+3}{w_{4j-4}}-q_{2^{2\alpha _{k}+1}-4j+1}{w_{4j-4}}\right) \right\vert  \\
&\geq & q_1-q_3-\sum_{j=2^{2\alpha _{k}-2}+1}^{2^{2\alpha
			_{k}-1}}\left\vert q_{2^{2\alpha _{k}+1}-4j+3}-q_{2^{2\alpha
			_{k}+1}-4j+1}\right\vert  \\
&\geq&  q_1-q_3-\frac12(q_{3}-q_{2^{2\alpha _{k}}-1})\geq q_1-\frac{3}{2}q_3.
\end{eqnarray*}

The proof is complete.
\end{proof}

\section{The Main result}

% Our main result reads:
In previous Sections we have discussed a number of inequalities and sometimes their sharpness.
Our main result is the following new sharpness result:
\begin{theorem}
Let $0\leq \alpha\leq 1$, $\beta$ be any non-negative real number and $t_n$ be N\"{o}rlund means with convex and non-increasing sequence $\left\{ q_{k}:k\in \mathbb{N}\right\}$ satisfying the condition
\begin{equation}\label{jig}
\frac{q_1-({3}/{2})q_3}{Q_n}\geq \frac{C}{ n^\alpha\log^{\beta}n},
\end{equation}
for some positive constant $C.$ Then, for any $0<p<1/(1+\alpha)$ there exists a martingale $f\in H_{p}$ such that 
\begin{equation*}
\sup_{n\in \mathbb{N}} \left\Vert t_{2^n}f\right\Vert_{weak-L_p}=\infty.
\end{equation*}
\end{theorem}

\begin{proof} 
Let $0<p<1/(1+\alpha)$. Under condition \eqref{jig} there exists a sequence $\left\{ n _{k}:k\in \mathbb{N}\right\}$ such that
$$\frac{2^{2n_{k}(1/p-1)}}{n_{k}Q_{2^{2n_k+1}}}\geq \frac{2^{2n_{k}(1/p-1-\alpha)}}{n_k^{\beta+1}}\to \infty, \ \ \text{as} \ \ k\to\infty.$$
Let $\left\{ \alpha _{k}:k\in \mathbb{N}\right\}\subset \left\{ n _{k}:k\in \mathbb{N}\right\} $ be an increasing sequence of
positive integers such that
\begin{equation}
\sum_{k=0}^{\infty }\alpha _{k}^{-p/2}<\infty ,  \label{3}
\end{equation}
\begin{equation}
\sum_{\eta =0}^{k-1}\frac{\left(2^{2\alpha_{\eta}}\right)^{1/p}}{\sqrt{
\alpha_{\eta}}}<\frac{\left(2^{2\alpha_{k}}\right)^{1/p}}{\sqrt{\alpha
_{k}}}  \label{4}
\end{equation}
and
\begin{equation}\label{5}
\frac{\left( 2^{2\alpha _{k-1}}\right) ^{1/p}}{\sqrt{\alpha _{k-1}}}<\frac{q_1-q_3-({3}/{2})q_5}{Q_{2^{2\alpha _{k}+1}}}\frac{2^{2\alpha
		_{k}(1/p-1)-3}}{\alpha _{k}}.  
\end{equation}

Let 
\begin{equation*}
f^{\left( n\right) } :=\sum_{\left\{ k;\text{ }2\alpha
_{k}<n\right\} }\lambda _{k}a_{k},
\end{equation*}
where 
\begin{equation*}
\lambda _{k}=\frac{1}{\sqrt{\alpha _{k}}} \ \ \ \text{ and } \ \ \ 
a_{k}={2^{2\alpha _{k}(1/p-1)}}\left( D_{2^{2\alpha _{k}+1}}-D_{2^{2\alpha
_{k}}}\right) .
\end{equation*}

From (\ref{3}) and Lemma \ref{Weisz} we find that $f \in H_{p}.$ 

It is easy to prove that
\begin{equation}
\widehat{f}(j)=\left\{ 
\begin{array}{l}
\frac{2^{2\alpha _{k}(1/p-1)}}{\sqrt{\alpha _{k}}},\,\,\text{ if \thinspace
\thinspace }j\in \left\{ 2^{2\alpha _{k}},...,\text{ ~}2^{2\alpha
_{k}+1}-1\right\} ,\text{ }k\in \mathbb{N}, \\ 
0,\text{ \thinspace \thinspace \thinspace if \thinspace \thinspace
\thinspace }j\notin \bigcup\limits_{k=1}^{\infty }\left\{ 2^{2\alpha
_{k}},...,\text{ ~}2^{2\alpha _{k}+1}-1\right\}\text{ .}
\end{array}%
\right.  \label{8}
\end{equation}

Moreover,
\begin{eqnarray} \label{10a}
&&t_{2^{2\alpha _{k}+1}}f \\ \notag
&=&\frac{1}{Q_{2^{2\alpha _{k}+1}}}\sum_{j=1}^{2^{2\alpha _{k}}-1}q_{2^{2\alpha _{k}+1}-j}{S_{j}f}+\frac{1}{Q_{2^{2\alpha _{k}+1}}}%
\sum_{j=2^{2\alpha _{k}}}^{2^{2\alpha _{k}+1}-1}q_{{2^{2\alpha
_{k}+1}-j}} {S_{j}f} \\
&:=&I+II.  \notag
\end{eqnarray}

Let $j<2^{2\alpha _{k}}.$ By combining (\ref{4}), (\ref{5}) and (\ref{8}) we can conclude that
\begin{eqnarray*}
\left\vert S_{j}f\left( x\right) \right\vert 
&\leq &\sum_{\eta =0}^{k-1}\sum_{v=2^{2\alpha _{\eta }}}^{2^{2\alpha _{\eta
}+1}-1}\left\vert \widehat{f}(v)\right\vert \\
&\leq& \sum_{\eta
=0}^{k-1}\sum_{v=2^{2\alpha _{\eta }}}^{2^{2\alpha _{\eta }+1}-1}\frac{
2^{2\alpha _{\eta }(1/p-1)}}{\sqrt{\alpha _{\eta }}} \leq \sum_{\eta =0}^{k-1}\frac{2^{2\alpha _{\eta }/p}}{\sqrt{\alpha _{\eta}}} \leq \frac{2^{2\alpha _{k-1}/p+1}}{\sqrt{\alpha _{k-1}}}.
\end{eqnarray*}
Hence,
\begin{eqnarray} \label{11a}
\left\vert I\right\vert &\leq &\frac{1}{Q_{2^{2\alpha _{k}+1}}}\underset{j=1}
{\overset{2^{2\alpha _{k}}-1}{\sum }}q_{2^{2\alpha _{k}+1}-j}{\left\vert S_{j}f\left( x\right)
	\right\vert }  \\ \notag
&\leq &\frac{1}{Q_{2^{2\alpha _{k}+1}}}\frac{2^{2\alpha _{k-1}/p}}{\sqrt{%
\alpha _{k-1}}}\sum_{j=1}^{M_{2\alpha _{k}+1}-1}q_{j}\leq\frac{2^{2\alpha _{k-1}/p}}{\sqrt{\alpha _{k-1}}}.  
\end{eqnarray}

Let $2^{2\alpha _{k}}\leq j\leq 2^{2\alpha_{k}+1}-1.$ Since
\begin{eqnarray*}
S_{j}f &=&\sum_{\eta =0}^{k-1}\sum_{v=2^{2\alpha _{\eta }}}^{2^{2\alpha
_{\eta }+1}-1}\widehat{f}(v)w_{v}+\sum_{v=2^{2\alpha _{k}}}^{j-1}\widehat{f}%
(v)w_{v} \\
&=&\sum_{\eta =0}^{k-1}\frac{2^{{2\alpha _{\eta }}\left( 1/p-1\right) }}{%
\sqrt{\alpha _{\eta }}}\left( D_{2^{2\alpha _{\eta }+1}}-D_{2^{2\alpha
_{\eta }}}\right) +\frac{2^{{2\alpha _{k}}\left( 1/p-1\right) }}{\sqrt{\alpha _{k}}}\left(
D_{j}-D_{2^{{2\alpha _{k}}}}\right),
\end{eqnarray*}
for $II$ we can conclude that
\begin{eqnarray}\label{12a}
&&II=
\frac{1}{Q_{2^{2\alpha_{k}+1}}}\underset{j=2^{2\alpha _{k}}}{\overset{%
2^{2\alpha _{k}+1}}{\sum }}\ q_{2^{2\alpha _{k}+1}-j}\left( \sum_{\eta
=0}^{k-1}\frac{2^{2\alpha_{\eta }\left( 1/p-1\right) }}{\sqrt{\alpha_{\eta}}}\left( D_{2^{2\alpha_{\eta }+1}}-D_{2^{2\alpha_{\eta }}}\right)\right)  \\
&&\ \ \ \ +\frac{1}{Q_{2^{2\alpha _{k}+1}}}\frac{2^{2\alpha _{k}\left( 1/p-1\right) }}{%
\sqrt{\alpha_k}}\sum_{j=2^{2\alpha_k}}^{2^{2\alpha _{k}+1}-1}
q_{2^{2\alpha _{k}+1}-j}{\left(D_j-D_{2^{2\alpha_k}}\right)}. \notag
\end{eqnarray}

Let $x\in I_{2}(e_0+e_1)\in I_0\backslash I_1.$ According to that $\alpha _{0}\geq 1$ we get that $2\alpha_k\geq 2,$ for all $k\in \mathbb{N}$ and if use (\ref{1dn}) we get
that $D_{2^{2\alpha_k}}=0$ 
and if we use Lemma 2 we can also conclude that
\begin{eqnarray}\label{14a}
II&=&\frac{1}{Q_{2^{2\alpha _{k}+1}}}\frac{2^{2\alpha _{k}(1/p-1)}}{
\sqrt{\alpha _{k}}}\sum_{j=2^{2\alpha_k}}^{2^{2\alpha _{k}+1}-1}
q_{2^{2\alpha _{k}+1}-j}{D_j}\\ \notag
&\geq& \frac{q_1-(3/2)q_3}{Q_{2^{2\alpha _{k}+1}}%
}\frac{2^{2\alpha _{k}(1/p-1)}}{\sqrt{\alpha _{k}}}.
\end{eqnarray}%
By combining (\ref{5}), (\ref{10a})-(\ref{14a}) for $x\in I_{2}(e_{0}+e_{1})$ we have that
\begin{eqnarray*}
\left\vert t_{2^{2\alpha_k+1}}f\left(x\right)\right\vert &\geq& II-I \\
&\geq &\frac{q_1-({3}/{2})q_3}{Q_{2^{2\alpha _{k}+1}}}\frac{2^{2\alpha _{k}(1/p-1)}}
{\sqrt{\alpha _{k}}}-\frac{q_1-(3/2)q_3}{Q_{2^{2\alpha _{k}+1}}}\frac{2^{2\alpha
_{k}(1/p-1)-3}}{\alpha _{k}} \\
&\geq &\frac{q_1-(3/2)q_3}{Q_{2^{2\alpha _k+1}}}\frac{2^{2\alpha_k(1/p-1)-3}}
{\sqrt{\alpha_k}}\geq \frac{C2^{2\alpha _{k}(1/p-1-\alpha)-3}}{\left(\ln 2^{2\alpha_k+1}+1\right)^{\beta}\sqrt{\alpha_k}} \\
&\geq& \frac{C2^{2\alpha _{k}(1/p-1-\alpha)-3}}{{\alpha^{\beta+1} _{k}}}.
\end{eqnarray*}

Hence, we can conclude that
\begin{eqnarray*}
&&\left\Vert t_{2^{2\alpha _{k}+1}}f\right\Vert _{weak-L_{p}}  \\
&\geq &\frac{C2^{2\alpha _{k}(1/p-1-\alpha)-3}}{\alpha^{\beta+1}_k}\mu \left\{ x\in G:\left\vert t_{2^{2\alpha
_{k}+1}}f\right\vert \geq \frac{C2^{2\alpha _{k}(1/p-1)-3}}{\alpha^{\beta+1}_k}\right\} ^{1/p} \\
&\geq &\frac{C2^{2\alpha _{k}(1/p-1-\alpha)-3}}{\alpha^{\beta+1}_k}\mu \left\{ x\in I_{2}(e_{0}+e_{1}):\left\vert
t_{2^{2\alpha _{k}+1}}f\right\vert \geq \frac{C2^{2\alpha _{k}(1/p-1)-6}}{\alpha^{\beta+1}_k}\right\} ^{1/p} \\
&\geq &\frac{C2^{2\alpha _{k}(1/p-1-\alpha)-3}}{\alpha^{\beta+1}_k}(\mu \left(  I_{2}(e_{0}+e_{1})\right) )^{1/p} \\
&>&\frac{c2^{2\alpha _{k}(1/p-1-\alpha)}}{\alpha _{k}^{\beta+1}}\rightarrow \infty,
\text{ \ \ as \ \ }k\rightarrow \infty .
\end{eqnarray*}

The proof is complete.
\end{proof}

In a concrete case we get a result for Nörlund logarithmic means $\left\{ L_n \right\}$ proved in \cite{BPST}:

\begin{corollary}
	Let $0<p<1.$ Then there exists a martingale $f\in H_{p}$ such that 
	\begin{equation*}
	\sup_{n\in \mathbb{N}} \left\Vert L_{2^n}f\right\Vert_{weak-L_p}=\infty.
	\end{equation*}
\end{corollary}
\begin{proof}
		It is easy to show that
$$q_1-({3}/{2})q_3=\frac{1}{2}-\frac{3}{2}\cdot\frac{1}{4}=\frac{1}{8}>0,$$
%	$$q_1-q_3-({1}/{2})q_5=\frac{1}{2}-\frac{1}{4}-\frac{1}{12}>0,$$
and condition \eqref{jig} holds true for $\alpha=\beta=0.$
\end{proof}

We also get similar new result for the $V_n$ means:
\begin{corollary}
	Let $0<p<1.$ Then there exists a martingale $f\in H_{p}$ such that 
	\begin{equation*}
	\sup_{n\in \mathbb{N}} \left\Vert V_{2^n}f\right\Vert_{weak-L_p}=\infty.
	\end{equation*}
\end{corollary}
\begin{proof}
	It is easy to show that
	$$q_1-({3}/{2})q_3=\frac{1}{\ln 2}-\frac{3}{2}\cdot\frac{1}{\ln4}=\log_2^e-(3/2)\frac{\log_2^e}{\log_2^4}=\log_2^e\left(1-\frac{3}{4}\right)>0,$$
	%	$$q_1-q_3-({1}/{2})q_5=\frac{1}{2}-\frac{1}{4}-\frac{1}{12}>0,$$
	and condition \eqref{jig} holds true for $\alpha=\beta=0.$
\end{proof}

We also get a corresponding new result for the Ces\`aro means $\sigma^\alpha_{2^n}.$

\begin{corollary}
	Let $0<p<1/(1+\alpha),$ for some $ 0< \alpha\leq 0.56.$ Then there exists a martingale $f\in H_{p}$ such that 
	\begin{equation*}
	\sup_{n\in \mathbb{N}} \left\Vert \sigma^\alpha_{2^n}f\right\Vert_{weak-L_p}=\infty.
	\end{equation*}
\end{corollary}
\begin{proof}
By a routine calculation we find that
$$q_1-({3}/{2})q_3=\alpha-\frac{\alpha(\alpha+1)(\alpha+2)}{4}=\alpha\cdot\frac{2-3\alpha-\alpha^2}{4}.$$
%$$q_1-q_3-({1}/{2})q_5=\alpha\frac{136-\alpha^4-50\alpha^3-35\alpha^2-170\alpha}{240}.$$
%It is easy to show that $ 0< \alpha<0.638$ this expression is positive.
It is easy to show that when $ 0< \alpha<0.56$ this expression is positive.
Hence, condition \eqref{jig} holds true for $\beta=0$ and $0<\alpha<1.$
\end{proof}

\begin{corollary}
	Let $0<p<1/(1+\alpha),$ for some $ 0< \alpha\leq 0.41.$ Then there exists a martingale $f\in H_{p}$ such that 
	\begin{equation*}
	\sup_{n\in \mathbb{N}} \left\Vert U^\alpha_{2^n}f\right\Vert_{weak-L_p}=\infty.
	\end{equation*}
\end{corollary}
\begin{proof}
	By a straightforward calculation we find that
	$$q_1-({3}/{2})q_3=2^{1-\alpha}-({3}/{2})4^{1-\alpha}=2^{1-\alpha}\left(1-3/ 2^{2-\alpha}\right).$$
	%$$q_1-q_3-({1}/{2})q_5=\alpha\frac{136-\alpha^4-50\alpha^3-35\alpha^2-170\alpha}{240}.$$
	%It is easy to show that $ 0< \alpha<0.638$ this expression is positive.
	It is easy to show that when $ 0< \alpha<0.41$ this expression is positive. 
	Hence, condition \eqref{jig} holds true for $\beta=0$ and $0<\alpha<1.$
\end{proof}

\section{Open questions and final remarks}

%As we have mentioned in the introduction, Nörlund logarithmic means $\left\{ L_n \right\}$ are not bounded from $H_p(G)$ to the space $L_p(G)$ for $0<p\leq 1.$ On the other hand, if we consider subsequence $\left\{ L_{2^k}\right\},$ then they are bounded from $H_1(G)$ to  $H_1(G),$ but boundedness of $\left\{ L_{2^k}\right\}$ does not hold from the martingale Hardy space  $H_p(G)$ to the space weak-$L_p(G)$ for $0<p<1.$
%
%To analyze these results, in the endpoint case $p=1$ we have negative result for $\left\{ L_{n}\right\},$ but positive result for subsequence $\left\{ L_{2^k}\right\}.$
%
%The endpoint case for Ces\`aro means $\sigma_n^{\alpha}$ is $p=1/(1+\alpha).$ So the following interesting open problems arise immediately:

\begin{remark}
This article can be regarded as a complement of the new book \cite{PTWbook}. In this book also a number of open problems are raised. Also this new investigation implies some corresponding open questions.
\end{remark}

\textbf{Open Problem 1:} Let $0<p<1/(1+\alpha),$ for some $ 0.56< \alpha<1.$ Does there exist a martingale $f\in H_{p}$ such that 
\begin{equation*}
\sup_{n\in \mathbb{N}} \left\Vert \sigma^\alpha_{2^n}f\right\Vert_{weak-L_p}=\infty?
\end{equation*}

\textbf{Open Problem 2:} Let $0<p<1/(1+\alpha),$ for some $ 0.41< \alpha<1.$ Does there exist a martingale $f\in H_{p}$ such that 
\begin{equation*}
\sup_{n\in \mathbb{N}} \left\Vert U^\alpha_{2^n}f\right\Vert_{weak-L_p}=\infty?
\end{equation*}

We also can investigate similar problems for more general summability methods:

\textbf{Open Problem 3:} Let $0<p<1/(1+\alpha),$ for some $ 0.56< \alpha<1$  and $t_n$ be  Nörlund means of Walsh-Fourier series with non-increasing and convex sequence $\left\{ q_{k}:k\in \mathbb{N}\right\},$ satisfying the condition \eqref{jig}.

Does there exist a martingale $f\in H_{1/(1+\alpha)} (0<p<1),$ such that
$$
\sup_{n\in\mathbb{N}}\left\Vert t_{2^n}f\right\Vert _{H_{1/(1+\alpha)}}=\infty?
$$

\textbf{Open Problem 4:} Let $f\in H_{1/(1+\alpha)},$ where $ 0< \alpha<1.$ Does there exists an absolute constant $C_\alpha,$ such that the following inequality holds 
\begin{equation*}
\left\Vert \sigma^\alpha_{2^n}f\right\Vert_{1/(1+\alpha)}\leq C_\alpha \left\Vert f\right\Vert_{H_{1/(1+\alpha)}}?
\end{equation*}

\textbf{Open Problem 5:} Let $f\in H_{1/(1+\alpha)},$ where $ 0< \alpha<1.$ Does there exists an absolute constant $C_\alpha,$ such that the following inequality holds 
\begin{equation*}
\left\Vert U^\alpha_{2^n}f\right\Vert_{1/(1+\alpha)}\leq C_\alpha \left\Vert f\right\Vert_{H_{1/(1+\alpha)}}?
\end{equation*}

\textbf{Open Problem 6:} Let $f\in H_{1/(1+\alpha)},$ where $ 0< \alpha<1$ and $t_n$ be  Nörlund means of Walsh-Fourier series with non-increasing and convex sequence $\left\{ q_{k}:k\in \mathbb{N}\right\},$ satisfying the condition \eqref{jig}.
Does there exists an absolute constant $C_\alpha,$ such that the following inequality holds 
\begin{equation*}
\left\Vert t^\alpha_{2^n}f\right\Vert_{1/(1+\alpha)}\leq C_\alpha \left\Vert f\right\Vert_{H_{1/(1+\alpha)}}?
\end{equation*}

\begin{remark}
It is an important relation between Walsh-Fourier series and Wavelet theory, see e.g. \cite{PTWbook} and the papers \cite{FGK} and \cite{FLS}. This is of special interest also for applications as described in the recent PhD thesis of K. Tangrand \cite{Tang}.
\end{remark}

\textbf{ Availability of data and material}

Not applicable.

\textbf{Competing interests}

The authors declare that they have no competing interests.

\textbf{Funding}

The publication charges for this manuscript is supported by the publication fund at UiT The Arctic University of Norway under code IN-1096130. 

\textbf{Authors' contributions}

DB and GT gave the idea and initiated the writing of this paper. LEP and KT followed up this with some complementary ideas. All authors read and approved the final manuscript.

\textbf{Acknowledgements}

The work of George Tephnadze was supported by Shota Rustaveli National
Science Foundation grant FR-19-676. The publication charges for this article have been funded by a grant from the publication fund of UiT The Arctic University of Norway. 
%The authors also would like to thank the two referees for helpful suggestions. 

\textbf{Author details}

$^{1}$The University of Georgia, School of Science and Technology, 77a Merab Kostava St, Tbilisi 0128, Georgia.

$^{2}$Department of Computer Science and Computational Engineering, UiT The Arctic University of Norway, P.O. Box 385, N-8505, Narvik, Norway and Department of Mathematics and Computer Science, Karlstad University, 65188 Karlstad, Sweden.

$^{3}$Department of Computer Science and Computational Engineering, UiT-The Arctic University of Norway, P.O. Box 385, N-8505, Narvik, Norway.

$^{4}$The University of Georgia, School of Science and Technology, 77a Merab Kostava St, Tbilisi 0128, Georgia.

%\textbf{Acknowledgements}
%
%The authors would like to thank the referees for helpful suggestions.

\end{document}